\documentclass[12pt,reqno]{amsart}

\newcommand\version{May 31, 2022}

\usepackage{amsmath,amsfonts,amsthm,amssymb,amsxtra}
\usepackage{bbm} 

\newtheorem{theorem}{Theorem}

\newtheorem{lemma}[theorem]{Lemma}

\theoremstyle{definition}

\theoremstyle{remark}

\newtheorem{remark}[theorem]{Remark}

\renewcommand{\epsilon}{\varepsilon}

\renewcommand{\phi}{\varphi}
\newcommand{\R}{\mathbb{R}}

\newcommand{\Sph}{\mathbb{S}}

\newcommand{\Z}{\mathbb{Z}}

\DeclareMathOperator{\dist}{dist}

\begin{document}

\title[Degenerate stability of some Sobolev inequalities --- \version]{Degenerate stability of some Sobolev inequalities}

\author{Rupert L. Frank}
\address[Rupert L. Frank]{Mathematisches Institut, Ludwig-Maximilans Universit\"at M\"unchen, The\-resienstr.~39, 80333 M\"unchen, Germany, and Munich Center for Quantum Science and Technology, Schellingstr.~4, 80799 M\"unchen, Germany, and Mathematics 253-37, Caltech, Pasa\-de\-na, CA 91125, USA}
\email{r.frank@lmu.de}

\renewcommand{\thefootnote}{${}$} \footnotetext{\copyright\, 2022 by the author. This paper may be reproduced, in its entirety, for non-commercial purposes.}

\begin{abstract}
	We show that on $\mathbb S^1(1/\sqrt{d-2})\times\mathbb S^{d-1}(1)$ the conformally invariant Sobolev inequality holds with a remainder term that is the fourth power of the distance to the optimizers. The fourth power is best possible. This is in contrast to the more usual vanishing to second order and is motivated by work of Engelstein, Neumayer and Spolaor. A similar phenomenon arises for subcritical Sobolev inequalities on $\mathbb S^d$. Our proof proceeds by an iterated Bianchi--Egnell strategy.
\end{abstract}

\maketitle

\section{Introduction and main results}

\subsection{Motivation}

In a fundamental paper, Bianchi and Egnell \cite{BiEg} answer a question by Br\'ezis and Lieb \cite{BrLi} and show that the Sobolev inequality on $\R^d$ holds with a remainder term involving the distance to the optimizers. More precisely, for some $c_d>0$ and all $u\in\dot H^1(\R^d)$,
\begin{equation}
	\label{eq:bianchiegnell}
	\|\nabla u\|_2^2 - S_d \| u\|_{2d/(d-2)}^2 \geq c_d \inf_{Q\in\mathcal Q} \|\nabla (u-Q)\|_2^2 \,.
\end{equation}
Here $S_d$ denotes the optimal constant in the Sobolev inequality on $\R^d$ and $\mathcal Q$ the set of its optimizers. Importantly, the right side in \eqref{eq:bianchiegnell} involves the \emph{square} of the distance to the set of optimizers, and simple examples show that this is best possible, in the sense that the inequality does not hold with a right side equal to a constant times $\|\nabla u\|_{2}^{2-\alpha} \inf_{Q\in\mathcal Q} \|\nabla (u-Q)\|_2^\alpha$ for $\alpha<2$.

In the last two decades there has been an abundance of stability results for various functional inequalities. Examples include, for instance, isoperimetric inequalities \cite{FuMaPr,FiMaPr,CiLe,FiFuMaMiMo}, $L^p$-Sobolev inequalities \cite{CiFuMaPr,FiNe,Ne,FiZh}, fractional Sobolev inequalities \cite{ChFrWe}, Gagliardo--Nirenberg inequalities \cite{BoDoNaSi}, Brunn--Minkowski, concentration and rearrangement inequalities \cite{FiMaPr0,FiJe,FiMaMo,Ch,FrLi}, eigenvalue inequalities \cite{Na,CaFrLi,BrDPVe,KaNaPoSt,AlKrNe}, solutions to elliptic equations with critical exponents \cite{CiFiMa,FiGl,DeSuWe}, Young's inequality \cite{Ch1}, Hausdorff--Young inequality \cite{Ch2}, etc. Many of these works use strategies inspired by the paper of Bianchi--Egnell and in essentially all works (exceptions being \cite{FiZh,FiJe} and one version of a refined H\"older inequality in \cite{CaFrLi}) the remainder term is quadratic in the distance to the optimizers.

Our work is motivated by the recent paper \cite{EnNeSp} of Engelstein, Neumayer and Spolaor concerning a quantitative version of a Sobolev-type inequality in conformal geometry. We recall that given a closed manifold $M$ of dimension $d\geq 3$ and a class $\mathcal C$ of conformally equivalent metrics, there is a constant $Y(M,\mathcal C)>-\infty$ such that for all $g\in\mathcal C$ and all $u\in H^1(M)$,
$$
\mathcal E_g[u] \geq Y(M,\mathcal C) \|u\|_{L^{2d/(d-2)}(M,v_g)}^2 \,.
$$
Here
\begin{equation}
	\label{eq:defe}
	\mathcal E_g[u]  := \int_M \left( |\nabla_g u|_g^2 + \frac{d-2}{4(d-1)}R_g u^2 \right)dv_g
\end{equation}
with $R_g$ the scalar curvature of $(M,g)$. The quantities $(4(d-1)/(d-2))\mathcal E_g[u]$ and $\|u\|_{L^{2d/(d-2)}(M,v_g)}^{2d/(d-2)}$ have the geometric meaning of the total scalar curvature and the volume, respectively, of the metric $u^{4/(d-2)}g$. The main result of \cite{EnNeSp} is that, if $(M,\mathcal C)$ is not conformally equivalent to the round sphere, then there are constants $c>0$ and $\alpha\geq 2$, depending on $(M,\mathcal C)$, such that for all $0\leq u\in H^1(M)$,
$$
\mathcal E_g[u] - Y(M,\mathcal C) \|u\|_{L^{2d/(d-2)}(M,v_g)}^2 \geq c \inf_{Q\in\mathcal Q} \frac{\|u-Q\|_{H^1}^{\alpha}}{\|u\|_{H^1}^{\alpha-2}} \,.
$$
Remarkably, while generically (in a sense made precise in \cite{EnNeSp}) one can take $\alpha=2$, there are examples in any dimension $d\geq 3$ where one needs to take some $\alpha\geq 4$. The simplest of these examples is
\begin{equation}
	\label{eq:productmfd}
	M = \Sph^1(\tfrac{1}{\sqrt{d-2}}) \times \Sph^{d-1}(1)
\end{equation}
with its standard product metric. Here $\Sph^n(r)\subset\R^{n+1}$ denotes the $n$-dimensional sphere of radius $r>0$.

The proof in \cite{EnNeSp} proceeds via a \L ojasiewicz inequality and, as far as we see, does not easily provide a specific value of $\alpha$ for a given $(M,\mathcal C)$. Therefore we think it is of interest to determine the optimal $\alpha$ in the example \eqref{eq:productmfd}. It turns out that $\alpha=4$, so this provides one of the few examples of a stability estimate with an optimal, nonquadratic remainder term. 

We believe that the underlying phenomenon and our way of handling it is of some interest even beyond the concrete example \eqref{eq:productmfd}. The basic reason for why there is \emph{no quadratic stability} is that the minimizer is degenerate in the sense that there is a zero mode of the Hessian of the minimization problem that does not come from symmetries of the set of minimizers. The reason for why there is \emph{quartic stability} is that a secondary nondegeneracy condition is satisfied. We stress that this reason for degenerate stability is different from that in the case of the $L^p$-Sobolev inequality for $2<p<d$ \cite{FiZh}.

The way we deal with the zero mode of the Hessian and the secondary nondegeneracy condition can be thought of as an iterated Bianchi--Egnell strategy. Namely, while Bianchi and Egnell project on the nearest optimizer, we do the same, but then zoom further in and project on the nearest zero-mode of the Hessian. This argument bears some vague resemblance to how in \cite{FrKoKo} we handled an asymptotic minimization situation where the expected leading term vanishes. We have not encountered this kind of argument in the context of stability of functional inequalities and we hope that it will be of use in related problems. 

The argument, except for the verifcation of the secondary nondegeneracy condition, is of a general nature, but we refrain from trying to formulate it abstractly. Instead, we illustrate it in three different circumstances of increasing technical difficulty.


\subsection{Main results}

We fix $2<q<\infty$ and set 
$$
S := \frac{(2\pi)^2}{q-2} \,.
$$
Then, for all $u\in H^1(\R/\Z)$,
\begin{equation}
	\label{eq:sob1d}
	\int_0^1 \left( (u')^2 + S u^2\right)dt \geq S \left( \int_0^1 |u|^q\,dt \right)^{2/q}.
\end{equation}
The constants in this inequality are optimal and equality holds if and only if $u$ is constant. These facts are well-known and we provide references before Lemma \ref{firstdecomp}.

The following theorem answers the stability question for this inequality involving the $H^1$ distance to the set of optimizers, that is, the set of constant functions.

\begin{theorem}\label{main}
	Let $2<q<\infty$. Then there is a constant $c_q>0$ such that for all $u\in H^1(\R/\Z)$,
	\begin{align*}
		\int_0^1 \left( (u')^2 + S u^2\right)dt - S \left( \int_0^1 |u|^q\,dt \right)^{2/q}
		& \geq c_q\, \frac{ \left( \int_0^1 \!\left( \left(u'\right)^2 \! + \! S \left(u-\int_0^1 u\,ds \right)^2\right)dt \right)^2}{\int_0^1 \!\left( \left(u'\right)^2 \! + \! S u^2\right)dt}\,.
	\end{align*}
\end{theorem}

\emph{Remarks.} (a) Note that
$$
\int_0^1 \!\left( \left(u'\right)^2 \! + \! S \left(u- \textstyle{\int_0^1 u\,ds} \right)^2\right)dt
= \inf_{c\in\R} \int_0^1 \left( ((u-c)')^2 + S (u-c)^2\right)dt \,,
$$
so the right side in the theorem involves an $H^1$ distance of $u$ to the set of optimizers.\\ 
(b) The right side is the fourth power of the distance to the set of optimizers. In Remark \ref{optimality} we show that the power four is best possible.\\
(c) Just like in the proof of the Bianchi--Egnell inequality \eqref{eq:bianchiegnell} in \cite{BiEg}, we will argue by compactness and do not get a computable value of $c_q$.

\medskip

Our second result is a higher-dimensional version of Theorem \ref{main}. Let $d\geq 2$ and $2<q<2d/(d-2)$. Then, for all $u\in H^1(\Sph^d)$,
\begin{align}
	\label{eq:sobsubcrit}
	\int_{\Sph^d} \left( |\nabla u|^2 + \frac{d}{q-2} u^2\right)d\omega \geq \frac{d}{q-2}\, |\Sph^d|^{1-2/q} \left( \int_{\Sph^d} |u|^q\,d\omega \right)^{2/q}.
\end{align}
The constants in this inequality are optimal and equality holds if and only if $u$ is constant. We provide references for these facts before Lemma \ref{firstdecompsphere}. Here is the analogue of Theorem \ref{main} for this inequality.

\begin{theorem}\label{main3}
	Let $d\geq 2$ and $2<q<2d/(d-2)$. Then there is a constant $c_{d,q}>0$ such that for all $u\in H^1(\Sph^d)$,
	\begin{align*}
		& \int_{\Sph^d} \left( |\nabla u|^2 + \frac{d}{q-2} u^2\right)d\omega - \frac{d}{q-2}|\Sph^d|^{1-2/q} \left( \int_{\Sph^d} |u|^q\,d\omega \right)^{2/q} \\
		& \geq c_{d,q}\, \frac{ \left( \int_{\Sph^d} \left( |\nabla u|^2 + \frac{d}{q-2} \left(u-|\Sph^d|^{-1} \int_{\Sph^d} u\,d\omega \right)^2\right)d\omega \right)^2}{\int_{\Sph^d} \left( |\nabla u|^2 + \frac{d}{q-2} u^2\right)d\omega}\,.
	\end{align*}
\end{theorem}

\emph{Remarks.} The same remarks (a), (b) and (c) on Theorem \ref{main} are relevant here, too. Optimality is proved in Remark \ref{optimality3}.

\medskip

Our third and final result concerns the example \eqref{eq:productmfd}. In this case it is known and implicitly contained in Schoen's work \cite{Sc} (see Lemma \ref{firstdecompmfd} below) that for all $u\in H^1(M)$,
$$
\mathcal E_g[u] \geq Y \left( \int_M |u|^{2d/(d-2)}\,dv_g \right)^{(d-2)/d}
$$
with optimal constant
$$
Y := \frac{(d-2)^2}{4} \left( \frac{2\pi}{\sqrt{d-2}} |\Sph^{d-1}| \right)^{2/d} = \frac{(d-2)^2}{4} \left( \mathrm{Vol}_g(M) \right)^{2/d}.
$$
Moreover, equality is attained if and only if $u$ is a constant. Here $\mathcal E_g$ is as in \eqref{eq:defe} and we note that $R_g = (d-1)(d-2)$. Our stability result reads as follows.

\begin{theorem}\label{main2}
	Let $d\geq 3$ and let $M = \Sph^1(\tfrac{1}{\sqrt{d-2}}) \times \Sph^{d-1}(1)$ with its standard product metric. Then there is a constant $c_d>0$ such that for all $u\in H^1(M)$,
	$$
	\mathcal E_g[u] - Y \left( \int_M |u|^{2d/(d-2)}\,dv_g \right)^{(d-2)/d}
	\geq c_d\, \frac{\left( \mathcal E_g[u - (\mathrm{Vol}_g(M))^{-1} \int_M u\, dv_g] \right)^2}{\mathcal E_g[u]} \,.
	$$
\end{theorem}

\emph{Remarks.} The same remarks (a), (b) and (c) on Theorem \ref{main} are relevant here. In particular, since $R_g$ is a positive constant, $\mathcal E_g[u]$ is equivalent to $\|u\|_{H^1}^2$ and the infimum of $\mathcal E_g[u-c]$ over all $c\in\R$ is attained for $u= (\mathrm{Vol}_g(M))^{-1} \int_M u\, dv_g$. Optimality is proved in Remark \ref{optimality2}.

\medskip

The remainder of this paper consists of three sections, devoted to the proofs of Theorems \ref{main}, \ref{main2} and \ref{main3}, respectively. We will provide all the details in the first case and focus on the additional difficulties in the second and third case.


\subsection{Acknowledgements}

The author wishes to thank R.~Neumayer for several discussions on the topic of this paper and her seminar talk in January 2021 at Caltech which motivated this work. J.~Dolbeault's help with references is much appreciated. Partial support through US National Science Foundation grants DMS-1363432 and DMS-1954995 and through German Research Foundation grant EXC-2111- 390814868 is acknowledged.


\section{Proof of Theorem \ref{main}}

In this section we prove degenerate stability for the family of one-dimensional Sobolev inequalities. The basic idea of the proof will be an iterated Bianchi--Egnell strategy. We will work throughout in the real Hilbert space $H^1(\R/\Z)$ with the inner product derived from the norm
\begin{equation}
	\label{eq:norm}
	\| u \| := \left( \int_0^1 \left( (u')^2 + S u^2\right)dt \right)^{1/2}.
\end{equation}
This norm depends through $S$ on the fixed parameter $2<q<\infty$. We abbreviate
$$
\overline u := \int_0^1 u\,dt
$$
and we denote the $L^q$-norm on $\R/\Z$ by $\|u\|_q$.

Inequality \eqref{eq:sob1d} appears in an equivalent form involving an ultraspherical operator in the work of Bakry and \'Emery \cite[pp.~204--205]{BaEm}. Earlier, \cite[Appendix B]{GiSp} (see also \cite[Corollary 6.2]{BVVe}) considered the Euler--Lagrange equation of the higher dimensional analogue of \eqref{eq:sob1d}. Their argument, which works and, in fact, simplifies in the one-dimensional context, shows that equality holds only for constants; see also \cite{DoEsKoLo}. Inequality \eqref{eq:sob1d} also appears in \cite[Theorem 4]{Be}, where it is deduced from \cite{L}, and an inspection of its proof again shows that equality holds only for constants.

\begin{lemma}\label{firstdecomp}
	Let $(u_n)\subset H^1(\R/\Z)$ be a sequence with $\|u_n\|^2 = S$ and $\|u_n\|_q \to 1$. Then, along a subsequence,
	$$
	u_n = \lambda_n \left( 1 + r_n \right)
	$$
	where $\lambda_n\in\R$, $r_n\in H^1(\R/\Z)$ and, for a $\sigma\in\{+1,-1\}$,
	\begin{equation}
		\label{eq:decomp1prop}
		\lambda_n\to \sigma \,,
		\qquad
		\| r_n \|\to 0 \,,
		\qquad
		\int_0^1 r_n\,dt = 0 \,.
	\end{equation}
\end{lemma}

\begin{proof}
	Since $(u_n)$ is bounded in $H^1(\R/\Z)$, it is bounded in $C^{1/2}(\R/\Z)$ and therefore equicontinuous. Thus, after passing to a subsequence, $(u_n)$ converges weakly in $H^1$ and uniformly to a function $u\in H^1(\R/\Z)$. By lower semicontinuity, we have $\|u\|^2\leq S$ and, by uniform convergence, $\|u\|_q = \lim_{n\to\infty} \|u_n\|_q = 1$. Thus, necessarily, $\|u\|^2= S$ and $u_n$ converges strongly in $H^1(\R/\Z)$ to $u$. Moreover, $u$ is a minimizer in the Sobolev inequality and therefore, by the above discussion, $u$ is constant. Since $\|u\|_q=1$, we have $u=\sigma$ for a $\sigma\in\{+1,-1\}$. We now set
	$$
	\lambda_n := \overline{u_n} \,,
	\qquad
	r_n := \frac{u_n}{\overline{u_n}} - 1 \,.
	$$
	By the above mentioned convergence properties, $\lambda_n\to \sigma$ and $r_n\to 0$ in $H^1$.
\end{proof}

In what follows an important role is played by the function
$$
g(t) := \cos(2\pi t)
$$
and its translates. The reason for this is that $g$ is a zero mode of the Hessian of the minimization problem.

\begin{lemma}\label{seconddecomp}
	Let $(u_n)\subset H^1(\R/\Z)$ be a sequence with $\|u_n\|^2 = S$ and
	\begin{equation}
		\label{eq:seconddecompass}
		\frac{\|u_n\|^2 - S \|u_n\|_q^2}{\| u_n -\overline{u_n}\|^2} \to 0 \,.
	\end{equation}
	Then, along a subsequence,
	$$
	u_n = \lambda_n \left( 1 + \mu_n \left(g(\cdot - t_n) + R_n \right) \right)
	$$
	where $\lambda_n,\mu_n\in\R$, $t_n\in\R/\Z$, $R_n\in H^1(\R/\Z)$ and, for a $\sigma\in\{+1,-1\}$,
	\begin{equation*}
		\lambda_n\to \sigma \,,
		\qquad
		\mu_n \to 0 \,,
		\qquad
		\| R_n \|\to 0
	\end{equation*}
	and
	\begin{equation*}
		\int_0^1 R_n\,dt = \int_0^1 R_n \cos2\pi(t-t_n)\,dt = \int_0^1 R_n \sin 2\pi(t-t_n)\,dt = 0 \,.
	\end{equation*}
\end{lemma}

\begin{proof}
	Since $\| u_n -\overline{u_n}\|^2= \inf_{c\in\R} \| u_n -c\|^2\leq \|u_n\|^2 = S$, assumption \eqref{eq:seconddecompass} implies that $\|u_n\|_q \to 1$. Therefore the previous lemma is applicable and, along a subsequence, we can decompose $u_n = \lambda_n (1+r_n)$ as described there. 
	
	We now expand the terms in the Sobolev inequality to `quadratic order'. We use the fact that, uniformly for $\tau\in\R$,
	$$
	|1+ \tau|^q = 1+ q \tau + \frac12 q (q-1) \tau^2 + \mathcal O(|\tau|^{\min\{3,q\}} + |\tau|^q)\,.
	$$
	Thus,
	$$
	|u_n|^q = |\lambda_n|^q \left( 1 + q r_n + \frac12 q (q-1) r_n^2 + \mathcal O(|r_n|^{\min\{3,q\}} + |r_n|^q) \right)
	$$
	and
	$$
	\|u_n\|_q^q = |\lambda_n|^q \left( 1 + \frac12 q (q-1) \int_0^1 r_n^2\,dt + \mathcal O( \|r_n\|_q^{\min\{3,q\}}) \right).
	$$
	(Here we used the fact that $r_n$ has mean value zero and that $\|r_n\|_q\to 0$.) Thus,
	$$
	\|u_n\|_q^{2} = \lambda_n^{2} \left( 1 + (q-1) \int_0^1 r_n^2\,dt + \mathcal O( \|r_n\|_q^{\min\{3,q\}}) \right).
	$$
	On the other hand, again by the mean value zero property,
	$$
	\|u_n\|^2  = \lambda_n^2 \left( S + \|r_n\|^2 \right).
	$$
	Putting this together, we obtain
	\begin{equation}
		\label{eq:expansionquad}
		\|u_n\|^2 - S\|u_n\|_q^2 =  \lambda_n^2 \left( \int_0^1 \left( (r_n')^2 - S(q-2)r_n^2 \right)dt + \mathcal O( \|r_n\|_q^{\min\{3,q\}}) \right).
	\end{equation}
	
	Since $\| u_n -\overline{u_n}\|^2 = \lambda_n^2 \|r_n\|^2$, the expansion \eqref{eq:expansionquad} shows that assumption \eqref{eq:seconddecompass} is equivalent to
	$$
	\frac{\int_0^1 \left( (r_n')^2 - S(q-2)r_n^2 \right)dt}{\|r_n\|^2} \to 0 \,.
	$$
	The kernel of the quadratic form in the numerator is spanned by $g(t) = \cos 2\pi t$ and $\sin 2\pi t$. On the orthogonal complement of this kernel and the negative direction corresponding to constants the quadratic form is equivalent to $\|\cdot\|^2$. Thus, if we define
	$$
	\alpha_n := 2 \int_0^1 r_n \cos 2\pi t \,dt \,,
	\qquad
	\beta_n := 2 \int_0^1 r_n \sin 2\pi t\,dt \,,
	$$
	and $s_n$ by
	$$
	r_n = \alpha_n \cos 2\pi t + \beta_n \sin 2\pi t + s_n \,,
	$$
	then
	$$
	\int_0^1 s_n \,dt = \int_0^1 s_n \cos2\pi t\,dt = \int_0^1 s_n \sin 2\pi t\,dt = 0
	$$
	and
	$$
	\| r_n\|^2 = \frac12 \left( (2\pi)^2 + S \right) \left( \alpha_n^2 + \beta_n^2 \right) + \|s_n\|^2
	\qquad\text{and}\qquad
	\frac{\|s_n\|^2}{\|r_n\|^2} \to 0 \,.
	$$
	We set
	$$
	\mu_n := \sqrt{\alpha_n^2 + \beta_n^2} \,,
	\qquad
	R_n := \frac{s_n}{\sqrt{\alpha_n^2 + \beta_n^2}} \,.
	$$
	The fact that $\|r_n\|\to 0$ implies $\mu_n\to 0$ and the fact that $\|s_n\|/\|r_n\|\to 0$ implies $\|R_n\|\to 0$. Finally, we choose $t_n\in\R/\Z$ such that
	$$
	\frac{\alpha_n}{\sqrt{\alpha_n^2+ \beta_n^2}} \cos 2\pi t + \frac{\beta_n}{\sqrt{\alpha_n^2+ \beta_n^2}} \sin 2\pi t = \cos 2\pi(t-t_n) = g(t-t_n)
	$$
	and obtain the claimed decomposition.
\end{proof}

\begin{lemma}\label{seconddecompenergy}
	Let $(u_n)\subset H^1(\R/\Z)$ be a sequence with $\|u_n\|^2 = S$ and $\|u_n\|_q \to 1$. Then
	\begin{equation}
		\label{eq:seconddecompenergy}
		\liminf_{n\to\infty} \frac{\| u_n\|^2 \left( \|u_n\|^2 - S\|u_n\|_q^2\right)}{\| u_n -\overline{u_n}\|^4} \geq \frac{(q+2)(q-2)}{12\,(q-1)} \,.
	\end{equation}
\end{lemma}

The key point of this lemma is that the right side of \eqref{eq:seconddecompenergy} is strictly positive. While the precise value of the constant is not important for the proof of Theorem \ref{main}, we will show in Remark \ref{optimality} that it is best possible.

\begin{proof}
	\emph{Step 1.} We pass to a subsequence along which the liminf in \eqref{eq:seconddecompenergy} is realized. By Lemma \ref{firstdecomp} and its proof, $\|u_n - \overline{u_n} \| \to 0$. Therefore, if $\liminf_{n\to\infty} \left( \|u_n\|^2- S \|u_n\|_q^2 \right)/$ $\| u_n - \overline{u_n} \|^2 >0$, then the left side of \eqref{eq:seconddecompenergy} is equal to $+\infty$. Thus, in the following we assume that $\liminf_{n\to\infty} \left( \|u_n\|^2- S \|u_n\|_q^2 \right)/ \| u_n -\overline{u_n}\|^2 =0$.
	
	By Lemma \ref{seconddecomp}, after passing to a subsequence, we can write
	$$
	u_n = \lambda_n \left( 1 + \mu_n \left( g(\cdot - t_n) + R_n \right) \right),
	$$
	where $\lambda_n$, $\mu_n$, $t_n$ and $R_n$ are as in that lemma. By translation invariance, we may also assume that $t_n=0$.
	
	\medskip
	
	\emph{Step 2.} We now expand the terms in the Sobolev inequality to `quartic order'. We use the fact that for all $\tau\in[1/2,3/2]$, say,
	$$
	| 1 + \tau|^q = 1 + q \tau + \frac12 q(q-1)\tau^2 + \frac16 q(q-1)(q-2) \tau^3 + \frac1{24} q(q-1)(q-2)(q-3) \tau^4 + \mathcal O(\tau^5) \,.
	$$
	Since $\mu_n(g+R_n)$ tends to zero in $H^1(\R/\Z)$ and therefore in $L^\infty$, for all sufficiently large $n$, we have $|\mu_n(g+R_n)|\leq 1/2$ and therefore the above bound is applicable. Recalling the orthogonality conditions, we obtain
	\begin{align*}
		\| u_n\|_q^q & = |\lambda_n|^q \left( 1 + \frac12 q(q-1) \mu_n^2 \left(\|g\|_2^2 + \|R_n\|_2^2\right) + \frac12 q(q-1)(q-2) \mu_n^3 \int_0^1 g^2 R_n\,dt \right. \\
		& \qquad\qquad \left. + \frac1{24} q(q-1)(q-2)(q-3) \mu_n^4 \|g\|_4^4 + \mathcal O( |\mu_n|^3 \| R_n\|^2 + |\mu_n|^5) \right).
	\end{align*}
	Here we estimated, using the Schwarz inequality,
	$$
	\left| \mu_n^4 \int_0^1 g^3 R_n\,dt \right| = \mathcal O( |\mu_n|^3 \| R_n\|^2 + |\mu_n|^5) \,.
	$$
	Consequently,
	\begin{align*}
		\| u_n\|_q^{2} & = \lambda_n^{2} \left( 1 + (q-1) \mu_n^2 \left(\|g\|_2^2 + \|R_n\|_2^2\right) + (q-1)(q-2) \mu_n^3 \int_0^1 g^2 R_n\,dt \right. \\
		& \qquad\qquad + \frac1{12} (q-1)(q-2)(q-3) \mu_n^4 \|g\|_4^4 - \frac14(q-2)(q-1)^2 \mu_n^4 \|g\|_2^4	\\
		& \qquad\qquad \left. +\ \mathcal O( |\mu_n|^3 \| R_n\|^2 + |\mu_n|^5) \right).
	\end{align*}
	On the other hand, because of the orthogonality conditions,
	$$
	\| u_n \|^2 = \lambda_n^2 \left( S + \mu_n^2 \|g\|^2 + \mu_n^2 \|R_n\|^2 \right).
	$$
	Putting this together, we obtain
	\begin{align*}
		& \lambda_n^{-2} \left( \|u_n\|^2 - S\|u_n\|_q^2 \right) = \mu_n^2 \left( \|g\|^2 - S (q-1) \|g\|_2^2 \right) \\
		& \quad + \mu_n^2 \left( \|R_n\|^2 - S (q-1) \|R_n\|_2^2
		- S (q-1)(q-2) \mu_n \int_0^1 g^2 R_n\,dt \right) \\
		& \quad + \mu_n^4 \left( \frac14 S (q-2)(q-1)^2 \|g\|_2^4 - \frac1{12} S (q-1)(q-2)(q-3) \|g\|_4^4 \right) \\
		& \quad + \mathcal O( |\mu_n|^3 \| R_n\|^2 + |\mu_n|^5) \,.
	\end{align*}
	Using
	$$
	\int_0^1 g^2\,dt = \frac12 \,,
	\qquad
	\int_0^1 g^4 \,dt = \frac38 \,,
	$$
	we can simplify this expansion to
	\begin{align*}
		\lambda_n^{-2} \left( \|u_n\|^2 - S\|u_n\|_q^2 \right) & = \mu_n^2 \left( \! \|R_n\|^2 \!- S (q-1) \|R_n\|_2^2
		- \! S (q-1)(q-2) \mu_n \int_0^1 \! g^2 R_n\,dt \right) \\
		& \quad + \mu_n^4\, \frac{(q+1)(q-1)(q-2) }{32}\, S +\ \mathcal O( |\mu_n|^3 \| R_n\|^2 + |\mu_n|^5) \,.
	\end{align*}
	
	\medskip
	
	\emph{Step 3.} It remains to get a lower bound on the term that is quadratic plus linear in $R_n$. We expand $R_n$ into a Fourier series,
	$$
	R_n(t) = \sum_{k=2}^\infty \left( a_k \cos 2\pi k t + b_k \sin 2\pi k t \right). 
	$$
	(For notational simplicity, we do not reflect the dependence of the $a_k$ and $b_k$ on $n$.) Note that by the orthogonality conditions there are no terms involving $a_0$, $a_1$ or $b_1$.
	We have
	$$
	\int_0^1 (R_n')^2\,dt = \frac12 \sum_{k=2}^\infty (2\pi k)^2 \left( a_k^2 + b_k^2 \right)
	\qquad \text{and} \qquad
	\int_0^1 g^2 R_n \,dt = \frac14 a_2 \,.
	$$
	Therefore,
	\begin{align*}
		& \|R_n\|^2 - S (q-1) \|R_n\|_2^2 - S (q-1)(q-2) \mu_n \int_0^1 g^2 R_n\,dt - C |\mu_n| \|R_n\|^2 \\
		& = \frac12 \sum_{k=2}^\infty \left( (2\pi k)^2 - S(q-2) \right) \left( a_k^2 + b_k^2 \right) - S(q-1)(q-2) \frac14 \mu_n a_2  \\
		& \quad - \frac C2 |\mu_n| \sum_{k=2}^\infty \left( (2\pi k)^2 + S \right) \left( a_k^2 + b_k^2 \right) \\
		& = \frac{q-2}2 S \left( \sum_{k=2}^\infty \left( k^2 - 1 \right) \left( a_k^2 + b_k^2 \right) - \frac{q-1}2 \mu_n a_2 \right. \\
		& \qquad\qquad\quad \left. - C |\mu_n| \sum_{k=2}^\infty \left( k^2 + \frac1{q-2} \right) \left( a_k^2 + b_k^2 \right) \right) \\
		& = \frac{q-2}2 S \left( \left( \left( 3 - C |\mu_n|\left( 4 + \frac1{q-2} \right) \right) a_2^2 - \frac{q-1}2 \mu_n a_2 \right) \right. \\
		& \qquad\qquad\quad + \left(3 - C |\mu_n|\left( 4 + \frac1{q-2} \right) \right) b_2^2 \\
		& \qquad\qquad\quad \left. + \sum_{k=3}^\infty \left( \left( k^2 - 1 \right) - C|\mu_n| \left( k^2 + \frac1{q-2}\right) \right) \left( a_k^2 + b_k^2 \right) \right).
	\end{align*}
	Since $\mu_n\to 0$, we have for $n$ large enough, uniformly in $k\geq 2$,
	$$
	\left( k^2 - 1 \right) - C|\mu_n| \left( k^2 + \frac1{q-2}\right) > 0 \,.
	$$
	Under this assumption and abbreviating
	$$
	\rho_n := 3 - C |\mu_n|\left( 4 + \frac1{q-2} \right)>0 \,,
	$$
	we can bound
	\begin{align*}
		& \|R_n\|^2 - S (q-1) \|R_n\|_2^2 - S (q-1)(q-2) \mu_n \int_0^1 g^2 R_n\,dt - C |\mu_n| \|R_n\|^2 \\
		& \geq \frac{q-2}2 S \left( \rho_n a_2^2 - \frac{q-1}2 \mu_n a_2 \right)  
		= \frac{q-2}2 S \rho_n \left( \left( a_2 - \frac{q-1}{4\,\rho_n} \mu_n \right)^2 - \frac{(q-1)^2}{16 \,\rho_n^2} \mu_n^2 \right)	\\
		& \geq - \frac{(q-1)^2(q-2)}{32 \,\rho_n} S \mu_n^2 = - \frac{(q-1)^2(q-2)}{96} S \mu_n^2 + \mathcal O(|\mu_n|^3) \,.
	\end{align*}
	
	To summarize, we have shown that
	\begin{align*}
		\lambda_n^{-2} \left( \|u_n\|^2 - S\|u_n\|_q^2 \right)
		& \geq S \mu_n^4 \left( \frac{(q+1)(q-1)(q-2)}{32} - \frac{(q-1)^2(q-2)}{96} \right) + \mathcal O(|\mu_n|^5) \\
		& = S \mu_n^4\, \frac{(q+2)(q-1)(q-2)}{48} + \mathcal O(|\mu_n|^5) \,.
	\end{align*}

	On the other hand, we have, by the orthogonality conditions,
	\begin{align*}
		\mu_n^4 & = \frac{\|u_n - \overline{u_n} \|^4}{\lambda_n^4 (\|g\|^2 + \|R_n\|^2)^2}  
		= \frac{4}{(q-1)^2 S^2}\, \|u_n - \overline{u_n}\|^4  \left( 1+ o(1) \right).
	\end{align*}
	Inserting this into the previous bound, we get the claimed asymptotic inequality.
\end{proof}

\begin{remark}\label{optimality}
	The bound in Lemma \ref{seconddecompenergy} is best possible, both with respect to the power four and with respect to the constant on the right side. Indeed, it is saturated as $\epsilon\to 0$ for $u_\epsilon = 1 + \epsilon g + \epsilon^2 h$ with $h(t) := ((q-1)/12)\cos 4\pi t$. In the notation of the previous proof, this corresponds to $\mu_\epsilon=\epsilon$ and $R_\epsilon = \epsilon (h+o(1))$. The function $h$ is chosen in such a way that the square that is completed in the previous proof (Step 3) vanishes to leading order.	
\end{remark}

We are finally in position to prove our first main result.

\begin{proof}[Proof of Theorem \ref{main}]
	We argue by contradiction and assume that for some fixed $2<q<\infty$, no such $c_q>0$ exists. Then there is a sequence $(u_n)\subset H^1(\R/\Z)$ such that
	\begin{equation}
		\label{eq:contradictionass}
		\frac{\| u_n\|^2 \left( \|u_n\|^2 - S\|u_n\|_q^2\right)}{\| u_n - \overline{u_n} \|^4} \to 0 \,.
	\end{equation}
	By homogeneity we may assume that $\|u_n\|^2 = S$, which implies $\|u_n\|_q \leq 1$.
	
	Using $\|u_n-\overline{u_n}\|^2 = \inf_c \|u_n-c\|^2 \leq \|u_n\|^2=S$ we obtain
	$$
	\liminf_{n\to\infty} \frac{\| u_n\|^2 \left( \|u_n\|^2 - S\|u_n\|_q^2\right)}{\| u_n - \overline{u_n} \|^4} \geq \liminf_{n\to\infty} \left( 1 - \| u_n \|_q^2 \right).
	$$
	Combining this with \eqref{eq:contradictionass}, we deduce that $\|u_n\|_q \to 1$. Therefore, Lemma \ref{seconddecompenergy} is applicable and yields \eqref{eq:seconddecompenergy}, which contradicts \eqref{eq:contradictionass}.	
\end{proof}

Let us briefly review the previous proof and emphasize its main aspects. Lemma \ref{firstdecomp} is a standard ingredient in a Bianchi--Egnell-type proof. It decomposes a sequences as an optimizer plus a small remainder. Lemma \ref{seconddecomp} is an iteration of this, where now the remainder $r_n$ is decomposed as a main term, namely a zero mode, plus a secondary remainder $R_n$. The proof follows again the Bianchi--Egnell strategy of expanding to second order, but the crucial difference now is that the linear operator that appears has a kernel that is not due to symmetries of the set of optimizers. In Lemma \ref{seconddecompenergy} we expand the `energy' to fourth order. The key step is the completion of the square, which determines the leading order of the remainder $R_n$ in terms of the zero mode. This is the function $h$ in Remark \ref{optimality}. The problem-specific aspect of this proof is that to order $\mu_n^2$, the `energy gain' by introducing $R_n$, namely, $S(q-1)^2(q-2)/96$ is strictly smaller than the `energy loss' due to presence of $g$, namely, $S(q+1)(q-1)(q-2)/32$. We think of this as a \emph{secondary nondegeneracy condition}. By the validity of the Sobolev inequality, we know that the gain is not larger than the loss. Since it is strictly smaller, we obtain a stability inequality with a quartic remainder. If the secondary nondegeneracy condition would not be satisfied and we had equality, we could try to iterate again and to expand further. From this point of view the \L ojasiewicz inequality in the work \cite{EnNeSp} says that this procedure stops after finitely many iterations.


\section{Proof of Theorem \ref{main2}}

For $d\geq 3$ we consider the manifold
$$
M = \Sph^1(\tfrac{1}{\sqrt{d-2}}) \times \Sph^{d-1}(1)
$$
with its standard metric. Since $R_g = (d-1)(d-2)$, we have
$$
\| u \|^2 := \mathcal E_g[u] = \int_0^{2\pi/\sqrt{d-2}} \int_{\Sph^{d-1}} \left( \left| \frac{\partial u}{\partial s} \right|^2 + |\nabla_{\Sph^{d-1}} u|^2 + \frac{(d-2)^2}{4} u^2 \right) d\omega\,ds  \,.
$$
We will abbreviate $q=2d/(d-2)$ and denote the $L^q(M,dv_g)$-norm by $\|u\|_q$.

We use intentionally the same symbols $\|\cdot\|$ and $\|\cdot\|_q$ as in the previous section. We hope that this rather underlines the common features of the proofs than creates confusion.

\begin{lemma}\label{firstdecompmfd}
	Let $(u_n)\subset H^1(M)$ be a sequence with $\|u_n\|^2 = Y$ and $\|u_n\|_q \to 1$. Then, along a subsequence,
	$$
	u_n = \lambda_n \left( 1 + r_n \right)
	$$
	where $\lambda_n\in\R$, $r_n\in H^1(M)$ and, for a $\sigma\in\{+1,-1\}$,
	\begin{equation*}
		\lambda_n\to \sigma\, (\mathrm{Vol}_g(M))^{-1/q} \,,
		\qquad
		\| r_n \|\to 0 \,,
		\qquad
		\int_M r_n\,dv_g = 0 \,.
	\end{equation*}
\end{lemma}

This lemma can essentially be considered as known. Let us show how it can be deduced from results in the literature.

\begin{proof}
	We get an upper bound on the Yamabe constant by taking a constant trial function. The resulting upper bound is strictly small than the Sobolev constant on the sphere or equivalently on $\R^d$, namely $S_d$ in \eqref{eq:bianchiegnell}. Consequently, Lions's theorem \cite[Theorem 4.1]{Li} is applicable and yields relative compactness in $H^1(M)$ of minimizing sequences. (Note the typo of the statement in \cite[Theorem 4.1]{Li}; the relative compactness requires a strict `binding' inequality.) In particular, there is a minimizer. By general arguments, any minimizer is either nonnegative or nonpositive. Without loss of generality, we can restrict ourselves to nonnegative minimizers.
	
	To complete the proof of the lemma, we need to show that the only minimizers are constants. We consider the Euler--Lagrange equation satisfied by a minimizer and follow Schoen \cite{Sc}. By the maximum principle any nonnegative, nontrivial solution of the Euler--Lagrange equation is positive. Then, as shown in \cite{CaGiSp} using the moving plane method, any positive solution depends only on the variable $s$. Now an ODE analysis shows that the only positive solutions are constants. It is at this last step that the value $1/\sqrt{d-2}$ of the radius of the sphere enters.
\end{proof}

Compared to the previous section, we slightly change the definition of $g$. Now it denotes the function, depending only on the coordinate $s$ in the first factor of $M$,
$$
g(s) := \cos(\sqrt{d-2}\, s) \,.
$$

\begin{lemma}\label{seconddecompmfd}
	Let $(u_n)\subset H^1(M)$ be a sequence with $\|u_n\|^2 = Y$ and
	\begin{equation*}
		\frac{\|u_n\|^2 - Y \|u_n\|_q^2}{\| u_n -\overline{u_n}\|^2} \to 0 \,.
	\end{equation*}
	Then, along a subsequence,
	$$
	u_n = \lambda_n \left( 1 + \mu_n \left(g(\cdot - s_n) + R_n \right) \right)
	$$
	where $\lambda_n,\mu_n\in\R$, $s_n\in\R/( \tfrac{2\pi}{\sqrt{d-2}}\Z)$, $R_n\in H^1(M)$ and, for a $\sigma\in\{+1,-1\}$,
	\begin{equation*}
		\lambda_n\to \sigma\, (\mathrm{Vol}_g(M))^{-1/q} \,,
		\qquad
		\mu_n \to 0 \,,
		\qquad
		\| R_n \|\to 0
	\end{equation*}
	and
	\begin{equation*}
		\int_M R_n\,dv_g = \int_M R_n \cos\sqrt{d-2}(s-s_n)\,dv_g = \int_M R_n \sin \sqrt{d-2}(s-s_n)\,dv_g = 0 \,.
	\end{equation*}
\end{lemma}

\begin{proof}
	The proof of this lemma is essentially the same as that of Lemma \ref{seconddecomp}. The relevant quadratic form is now
	\begin{align*}
		& \int_M \left( |\nabla_g r|_g^2 + \frac{(d-2)^2}{4} r^2\right)dv_g - (q-1) Y \left( \mathrm{Vol}_g(M) \right)^{-1+2/q} \int_M r^2\,dv_g \\
		& = \int_M \left( |\nabla_g r|_g^2 - (d-2) r^2\right)dv_g \,.
	\end{align*}
	Its kernel is spanned by $g(s) = \cos(\sqrt{d-2}\,s)$ and $\sin(\sqrt{d-2}\,s)$. Therefore we can argue as before.	
\end{proof}

\begin{lemma}\label{seconddecompenergymfd}
	Let $(u_n)\subset H^1(M)$ be a sequence with $\|u_n\|^2 = Y$ and $\|u_n\|_q \to 1$. Then
	\begin{equation*}
		\liminf_{n\to\infty} \frac{\| u_n\|^2 \left( \|u_n\|^2 - Y\|u_n\|_q^2\right)}{\| u_n -\overline{u_n}\|^4} \geq \frac{(q+2)(q-2)}{12\,(q-1)} \,.
	\end{equation*}
\end{lemma}

\begin{proof}
	\emph{Step 1.} The proof for $d=3,4$ follows exactly the lines of that of Lemma \ref{seconddecompenergy}. Indeed, in these dimensions one has $q=2d/(d-2)\geq 4$ and therefore one can expand $|u_n|^q$ to fourth order even without using the $L^\infty$ convergence in Lemma \ref{seconddecompenergy}. For $d>4$, however, one has $q=2d/(d-2)<4$ and therefore the quartic expansion of $|u_n|^q$ is problematic. To overcome this issue, we first decompose $u_n$ as in Lemma \ref{seconddecompmfd} and then we further decompose
	$$
	R_n = S_n + T_n
	\qquad\text{with}\qquad
	S_n(s) := |\Sph^{d-1}|^{-1} \int_{\Sph^{d-1}} R_n(s,\omega)\,d\omega \,.
	$$
	The function $T_n$ has the property that for any function $\phi$ of $s$ alone,
	\begin{equation}
		\label{eq:meanzerot}
		\int_M \phi(s) T_n\,dv_g = 0 \,.
	\end{equation}
	By orthogonality,
	$$
	\| R_n \|^2 = \|S_n\|^2 + \|T_n\|^2 \,,
	$$
	so $\|R_n\|\to 0$ implies $\|S_n\|\to 0$ and consequently $S_n\to 0$ in $L^\infty$. This will allow us to argue for $S_n$ like we did in the proof of Lemma \ref{seconddecompenergy}. But first we need to get rid of the term $T_n$, and we do this by a spectral gap estimate. 
	
	\medskip
	
	\emph{Step 2.} Let us set (assuming without loss of generality that $s_n=0$)
	$$
	u_n = \tilde u_n + \lambda_n \mu_n T_n
	\qquad\text{with}\qquad
	\tilde u_n := \lambda_n \left( 1 + \mu_n (g + S_n) \right).
	$$
	Then, by a quadratic estimate as in the proofs of Lemmas \ref{seconddecomp} and \ref{seconddecompmfd},
	$$
	\| u_n\|_q^q = \|\tilde u_n\|_q^q + \frac12 q(q-1) \lambda_n^2 \mu_n^2 \int_M |\tilde u_n|^{q-2} T_n^2\,dv_g + \mathcal O( |\lambda_n|^q |\mu_n|^{\min\{3,q\}} \|T_n\|_q^{\min\{3,q\}}) \,.
	$$
	Note that the term linear in $T_n$ cancels by \eqref{eq:meanzerot} with $\phi = 
	|\tilde u_n|^{q-2} \tilde u_n$.	We also used the fact that $\|T_n\|_q\lesssim \|T_n\| \to 0$. Consequently,
	$$
	\| u_n\|_q^2 = \|\tilde u_n\|_q^2 + (q-1) \lambda_n^2 \| \tilde u_n\|_q^{-q+2} \mu_n^2 \int_M |\tilde u_n|^{q-2} T_n^2\,dv_g + \mathcal O( \lambda_n^2 |\mu_n|^{\min\{3,q\}} \|T_n\|_q^{\min\{3,q\}}) \,.
	$$
	
	In order to simplify the term quadratic in $T_n$, we need some rough expansions of $\tilde u_n$. Using $\|g+S_n\|_q \lesssim 1$ one finds without much effort that
	$$
	|\lambda_n|^{-q} \| \tilde u_n \|_q^q = \mathrm{Vol}_g(M)  + \mathcal O(|\mu_n|)
	$$
	and
	$$
	|\lambda_n|^{-q+2} \int_M |\tilde u_n|^{q-2} T_n^2\,dv_g = \int_M T_n^2\,dv_g
	+ \mathcal O(|\mu_n|^{\min\{1,q-2\}} \|T_n\|_q^2)
	$$
	Thus,
	$$
	\| \tilde u_n\|_q^{-q+2} \int_M |\tilde u_n|^{q-2} T_n^2\,dv_g = (\mathrm{Vol}_g(M))^{-1+2/q} \int_M T_n^2\,dv_g + \mathcal O(|\mu_n|^{\min\{1,q-2\}} \|T_n\|_q^2) \,.
	$$	
	On the other hand, because of the orthogonality conditions,
	\begin{equation}
		\label{eq:decompenergyorthomfd}
		\| u_n \|^2 = \| \tilde u_n \|^2 + \lambda_n^2 \mu_n^2 \|T_n\|^2 \,.
	\end{equation}
	Putting this together, we obtain
	\begin{align*}
		& \lambda_n^{-2} \left( \|u_n\|^2 - Y \| u_n\|_q^2 \right) = \lambda_n^{-2} \left( \|\tilde u_n\|^2 - Y \|\tilde u_n\|_q^2 \right) \\
		& \quad + \mu_n^2 \left( \|T_n\|^2  - (q-1) Y (\mathrm{Vol}_g(M))^{-1+2/q} \int_M T_n^2\,dv_g + \mathcal O(|\mu_n|^{\min\{1,q-2\}} \|T_n\|_q^2) \right).
	\end{align*}
	Just like in the proof of Lemma \ref{seconddecompmfd}, the term quadratic in $T_n$ involves the operator $-\Delta_g -(d-2)$. Since, by \eqref{eq:meanzerot}, $T_n$ is orthogonal to its kernel, which is spanned by $g(s) = \cos(\sqrt{d-2}\,s)$ and $\sin(\sqrt{d-2}\,s)$, and to its negative spectral subspace, which is spanned by the constant function, we have
	$$
	\|T_n\|^2  - (q-1) Y (\mathrm{Vol}_g(M))^{-1+2/q} \int_M T_n^2\,dv_g \gtrsim \|T_n\|^2
	$$
	with an implicit constant depending only on $d$. Thus, if $n$ is large enough, the error term $\mathcal O(|\mu_n|^{\min\{1,q-2\}} \|T_n\|_q^2)$ can be absorbed and we conclude that
	$$
	\|u_n\|^2 - Y \| u_n\|_q^2 \geq \|\tilde u_n\|^2 - Y \|\tilde u_n\|_q^2 \,.
	$$
	Moreover, we note that
	$$
	\| u_n - \overline{u_n}\|^2 = \| \tilde u_n - \overline{\tilde u_n} \|^2 + \lambda_n^2 \mu_n^2 \|T_n\|^2 \,.
	$$
	Since
	$$
	\| \tilde u_n - \overline{\tilde u_n} \|^2 = \lambda_n^2 \mu_n^2 \left( \|g\|^2 + \|S_n\|^2 \right) \geq \lambda_n^2 \mu_n^2 \|g\|^2 \,,
	$$
	and $\| T_n\|^2 \to 0$, we conclude that
	$$
	\| u_n - \overline{u_n}\|^2 = \| \tilde u_n - \overline{\tilde u_n} \|^2 \left( 1 + o(1)\right).
	$$
	Finally, by \eqref{eq:decompenergyorthomfd}, $\|u_n\| \geq \|\tilde u_n\|$. To summarize, we have shown that
	$$
	\frac{\| u_n\|^2 \left( \| u_n\|^2 - Y\| u_n\|_q^2\right)}{\|  u_n - \overline{ u_n}\|^4}
	\geq \frac{\| \tilde u_n\|^2 \left( \|\tilde u_n\|^2 - Y\|\tilde u_n\|_q^2\right)}{\| \tilde u_n - \overline{\tilde u_n}\|^4}  \left( 1 + o(1)\right).
	$$
	(With more effort one can show that the $o(1)$ error on the right side is not necessary, but we will not need this.)
	
	\medskip
	
	\emph{Step 3.} From this point on, the proof is exactly the same as that of Lemma \ref{seconddecompenergy}. In fact, one does not even have redo that argument, one can simply argue by scaling. Note that $\tilde u_n$ are functions depending only on the variable $s\in\Sph^1(\tfrac{1}{\sqrt{d-2}})$. If we set $\tilde u_n(s) = v_n(s\sqrt{d-2}/(2\pi))$, then $v_n$ is one-periodic and
	$$
	\frac{\| \tilde u_n\|^2 \left( \|\tilde u_n\|^2 - Y\|\tilde u_n\|_q^2\right)}{\| \tilde u_n - \overline{\tilde u_n}\|^4} = \frac{\| v_n\|^2 \left( \|v_n\|^2 - S\|v_n\|_q^2\right)}{\| v_n - \overline{v_n}\|^4} \,,
	$$
	where on the right side $\|\cdot\|$ stands for the norm \eqref{eq:norm} of functions in $H^1(\R/\Z)$ with $S = (2\pi)^2 /(q-2)$. The claimed bound now follows from that in Lemma \ref{seconddecompenergy}.	
\end{proof}

\begin{remark}\label{optimality2}
	The bound in Lemma \ref{seconddecompenergymfd} is best possible, both with respect to the power four and with respect to the constant on the right side. This follows from Remark \ref{optimality} by the same scaling as at the end of the previous proof.	
\end{remark}

Theorem \ref{main2} follows from Lemma \ref{seconddecompenergymfd} in the same way as Theorem \ref{main} follows from Lemma \ref{seconddecompenergy}. We omit the details.


\section{Proof of Theorem \ref{main3}}

We fix $d\geq 2$ and $2<q<2d/(d-2)$ and abbreviate, in this section,
$$
\| u \|^2 := \int_{\Sph^d} \left( |\nabla u|^2 + \frac{d}{q-2} u^2\right)d\omega
$$
and
$$
Y := \frac{d}{q-2}\,|\Sph^d|^{1-2/q} \,. 
$$
Moreover, $\|u\|_q$ will denote the $L^q$-norm on $\Sph^d$ and $\overline u = |\Sph^d|^{-1} \int_{\Sph^d} u\,d\omega$.

Let us briefly comment on the history of inequality \eqref{eq:sobsubcrit}. By symmetric decreasing rearrangment, it suffices to prove the inequality for functions that depend only on $\omega_{d+1}$ and the resulting inequality was shown in the work of Bakry and \'Emery \cite[pp.~204--205]{BaEm}. As mentioned before Lemma \ref{firstdecomp}, the inequality appears explicitly in the work of Bidaut-V\'eron and V\'eron \cite[Corollary 6.2]{BVVe}, who also show that equality holds only for constant. Their work builds upon \cite[Appendix B]{GiSp}. In addition, like \eqref{eq:sob1d}, inequality \eqref{eq:sobsubcrit} appears in \cite[Theorem 4]{Be}, from which one can also deduce the cases of equality.

\begin{lemma}\label{firstdecompsphere}
	Let $(u_n)\subset H^1(\Sph^d)$ be a sequence with $\|u_n\|^2 = Y$ and $\|u_n\|_q \to 1$. Then, along a subsequence,
	$$
	u_n = \lambda_n \left( 1 + r_n \right)
	$$
	where $\lambda_n\in\R$, $r_n\in H^1(\Sph^d)$ and, for a $\sigma\in\{+1,-1\}$,
	\begin{equation*}
		\lambda_n\to \sigma\,|\Sph^d|^{-1/q} \,,
		\qquad
		\| r_n \|\to 0 \,,
		\qquad
		\int_{\Sph^d} r_n\,d\omega = 0 \,.
	\end{equation*}
\end{lemma}

\begin{proof}
	The argument is the same as in the proof of Lemma \ref{firstdecomp}, except that one replaces the compactness theorem of Arzel\`a--Ascoli by Rellich's. We omit the details.
\end{proof}

\begin{lemma}\label{seconddecompsphere}
	Let $(u_n)\subset H^1(\Sph^d)$ be a sequence with $\|u_n\|^2 = Y$ and
	\begin{equation*}
		\frac{\|u_n\|^2 - Y \|u_n\|_q^2}{\| u_n -\overline{u_n}\|^2} \to 0 \,.
	\end{equation*}
	Then, along a subsequence,
	$$
	u_n = \lambda_n \left( 1 + \mu_n \left(e_n\cdot\omega + R_n \right) \right)
	$$
	where $\lambda_n,\mu_n\in\R$, $e_n\in\Sph^d$, $R_n\in H^1(M)$ and, for a $\sigma\in\{+1,-1\}$,
	\begin{equation*}
		\lambda_n\to \sigma\,|\Sph^d|^{-1/q} \,,
		\qquad
		\mu_n \to 0 \,,
		\qquad
		\| R_n \|\to 0
	\end{equation*}
	and, for all $j=1,\ldots,d+1$,
	\begin{equation*}
		\int_{\Sph^d} R_n\,d\omega = \int_{\Sph^d} R_n \omega_j\,d\omega = 0 \,.
	\end{equation*}
\end{lemma}

\begin{proof}
	The proof of this lemma is essentially the same as that of Lemmas \ref{seconddecomp} and \ref{seconddecompmfd}. The relevant quadratic form is now
	\begin{align*}
		& \int_{\Sph^d} \left( |\nabla r|^2 + \frac{d}{q-2} r^2\right)d\omega - (q-1) \frac{d}{q-2} \int_{\Sph^d} r^2\,d\omega 
		= \int_{\Sph^d} \left( |\nabla r|^2 - d r^2\right)d\omega \,.
	\end{align*}
	The kernel of this quadratic form is spanned by spherical harmonics of degree one, that is, by $\omega_1,\ldots,\omega_{d+1}$. Therefore we can argue as before.	
\end{proof}

\begin{lemma}\label{seconddecompenergysphere}
	Let $(u_n)\subset H^1(\Sph^d)$ be a sequence with $\|u_n\|^2 = Y$ and $\|u_n\|_q \to 1$. Then
	\begin{equation}
		\label{eq:seconddecompenergysphere}
		\liminf_{n\to\infty} \frac{\| u_n\|^2 \left( \|u_n\|^2 -Y \|u_n\|_q^2\right)}{\| u_n -\overline{u_n}\|^4} \geq \frac{(d+1)(q-2)(2d-q(d-2))}{2(d+2)(d+3)(q-1)} \,.
	\end{equation}
\end{lemma}

Note that the expression on the right side is positive since $q<2d/(d-2)$. Its vanishing for $q=2d/(d-2)$ if $d\geq 3$ is consistent with the fact that in the Bianchi--Egnell inequality \eqref{eq:bianchiegnell} (and in its equivalent sphere version), one takes the infimum over the $(d+2)$-dimensional manifold of optimizers, whereas for $q<2d/(d-2)$ we are taking the infimum only over the one-dimensional set of constants. Note also that the constant in \eqref{eq:seconddecompenergysphere} coincides with the corresponding expression in Lemma \ref{seconddecompenergy} for $d=1$.

\begin{proof}
	\emph{Step 1.} The proof is similar to those of Lemmas \ref{seconddecompenergy} and \ref{seconddecompenergysphere}. As in those proofs we can pass to a subsequence along which the liminf in \eqref{eq:seconddecompenergysphere} is realized and we may assume that $\liminf_{n\to\infty} \left( \|u_n\|^2 - Y \|u_n\|_q^2\right) / \| u_n -\overline{u_n}\|^2=0$.
	
	By Lemma \ref{seconddecompsphere}, after passing to a subsequence and after a rotation, we can write
	\begin{equation}
		\label{eq:seconddecompsphereproof}
		u_n = \lambda_n \left( 1 + \mu_n \left( g + R_n \right) \right),
	\end{equation}
	where $\lambda_n$, $\mu_n$ and $R_n$ are as in that lemma and $g(\omega)=\omega_{d+1}$.
	
	\medskip
	
	\emph{Step 2.} We now restrict ourselves to the simpler case where $d=2,3$ and $4\leq q<2d/(d-2)$. Then we can expand $|1+\tau|^q$ to fourth order in $\tau$ and obtain as in the proof of Lemma~\ref{seconddecompenergy}, recalling the orthogonality conditions,
	\begin{align*}
		\| u_n\|_q^q & = |\lambda_n|^q \left( |\Sph^d| + \frac12 q(q-1) \mu_n^2 \left(\|g\|_2^2 + \|R_n\|_2^2\right) + \frac12 q(q-1)(q-2) \mu_n^3 \int_{\Sph^d} g^2 R_n\,d\omega \right. \\
		& \qquad\qquad \left. + \frac1{24} q(q-1)(q-2)(q-3) \mu_n^4 \|g\|_4^4 + \mathcal O( |\mu_n|^3 \| R_n\|^2 + |\mu_n|^5) \right).
	\end{align*}
	Consequently,
	\begin{align*}
		\| u_n\|_q^{2} & = \lambda_n^{2} |\Sph^d|^{2/q} \left( 1 + (q-1) \mu_n^2 |\Sph^d|^{-1} \left(\|g\|_2^2 + \|R_n\|_2^2\right) \right. \\
		& \qquad\quad + (q-1)(q-2) \mu_n^3 |\Sph^d|^{-1} \int_{\Sph^d} g^2 R_n\,d\omega \\
		& \qquad\quad + \frac1{12} (q-1)(q-2)(q-3) \mu_n^4 |\Sph^d|^{-1} \|g\|_4^4 - \frac14(q-2)(q-1)^2 \mu_n^4 |\Sph^d|^{-2} \|g\|_2^4	\\
		& \qquad\quad \left. +\ \mathcal O( |\mu_n|^3 \| R_n\|^2 + |\mu_n|^5) \right).
	\end{align*}
	On the other hand, because of the orthogonality conditions,
	$$
	\| u_n \|^2 = \lambda_n^2 \left( \frac{d}{q-2}\,|\Sph^d| + \mu_n^2 \|g\|^2 + \mu_n^2 \|R_n\|^2 \right).
	$$
	Putting this together, we obtain
	\begin{align*}
		& \lambda_n^{-2} \left( \|u_n\|^2 - Y \|u_n\|_q^2 \right) = \mu_n^2 \left( \|g\|^2 - \frac{d(q-1)}{q-2} \|g\|_2^2 \right) \\
		& \quad + \mu_n^2 \left( \|R_n\|^2 - \frac{d(q-1)}{q-2} \|R_n\|_2^2
		- d (q-1) \mu_n \int_{\Sph^d} g^2 R_n\,d\omega \right) \\
		& \quad + \mu_n^4 \left( \frac d4 (q-1)^2 |\Sph^d|^{-1} \|g\|_2^4 - \frac d{12} (q-1)(q-3) \|g\|_4^4 \right) \\
		& \quad + \mathcal O( |\mu_n|^3 \| R_n\|^2 + |\mu_n|^5) \,.
	\end{align*}
	Using
	\begin{equation}
		\label{eq:sphcomp}
		\frac1d \int_{\Sph^d} |\nabla g|^2\,d\omega = \int_{\Sph^d} g^2\,d\omega = \frac1{d+1}|\Sph^d| \,,
	\qquad
	\int_{\Sph^d} g^4 \,d\omega = \frac3{(d+1)(d+3)}|\Sph^d| \,,
	\end{equation}
	we can simplify this expansion to
	\begin{align*}
		\lambda_n^{-2} \left( \|u_n\|^2 - Y \|u_n\|_q^2 \right) & = \mu_n^2 \left( \! \|R_n\|^2 \!- \frac{d(q-1)}{q-2} \|R_n\|_2^2
		- \! d (q-1) \mu_n \int_{\Sph^d} \! g^2 R_n\,d\omega \right) \\
		& \quad + \mu_n^4\, \frac{d(q-1)(q+d)}{2(d+1)^2(d+3)}\,|\Sph^d| +\ \mathcal O( |\mu_n|^3 \| R_n\|^2 + |\mu_n|^5) \,.
	\end{align*}
	
	\emph{Step 3.} It remains to get a lower bound on the term that is quadratic plus linear in $R_n$. We expand $R_n$ into spherical harmonics
	$$
	R_n(t) = \sum_{\ell=2}^\infty \sum_m a_{\ell,m} Y_{\ell,m} \,.
	$$
	Here, for each $\ell$, $(Y_{\ell,m})_m$ is an $L^2(\Sph^d)$-orthonormal basis of (real) spherical harmonics of degree $\ell$. The index $m$ runs through a finite set whose cardinality depends on $\ell$, but which will not be important for us. The only thing we will use is that the space of spherical harmonics of degree zero is spanned by constant functions and that of degree one by $\omega_1,\ldots,\omega_{d+1}$.
	
	For notational simplicity, we do not reflect the dependence of the $a_{\ell,m}$ on $n$. Note that by the orthogonality conditions there are no terms involving $\ell=0$ and $\ell=1$.	We have
	$$
	\int_{\Sph^d} |\nabla R_n|^2\,d\omega = \sum_{\ell=2}^\infty \sum_m \ell(\ell+d-1) a_{\ell,m}^2
	\qquad \text{and} \qquad
	\int_{\Sph^d} R_n^2\,d\omega = \sum_{\ell=2}^\infty \sum_m a_{\ell,m}^2 \,.
	$$
	Moreover, since $\omega_{d+1}^2 - 1/(d+1)$ is a spherical harmonic of degree two and since, by \eqref{eq:sphcomp},
	$$
	\int_{\Sph^d} (\omega_{d+1}^2 - 1/(d+1))^2\,d\omega = \frac{2d}{(d+1)^2 (d+3)}\,|\Sph^d| \,,
	$$
	we can assume, without loss of generality, that
	$$
	Y_{2,0}(\omega) = \sqrt{\frac{(d+1)^2(d+3)}{2d\,|\Sph^d|}} \left( \omega_{d+1}^2 - \frac 1{d+1} \right).
	$$
	Thus, since $\overline{R_n}=0$,
	$$
	\int_{\Sph^d} g^2 R_n \,d\omega = \int_{\Sph^d} (\omega_{d+1}^2 - 1/(d+1)) R_n \,d\omega
	= \sqrt{\frac{2d\,|\Sph^d|}{(d+1)^2(d+3)}} a_{2,0} \,.
	$$
	Therefore,
	\begin{align*}
		& \|R_n\|^2 - \frac{d(q-1)}{(q-2)} \|R_n\|_2^2 - d (q-1) \mu_n \int_{\Sph^d} g^2 R_n\,d\omega - C |\mu_n| \|R_n\|^2 \\
		& = \sum_{\ell=2}^\infty \sum_m \left( \ell(\ell+d-1) - d \right) a_{\ell,m}^2 - d(q-1) \mu_n \sqrt{\frac{2d\,|\Sph^d|}{(d+1)^2(d+3)}} a_{2,0}  \\
		& \quad - C |\mu_n| \sum_{\ell=2}^\infty \sum_m \left( \ell(\ell+d-1) + \frac d{q-2} \right) a_{\ell,m}^2 \\
		& = \left( \left( d+2 - C |\mu_n|\left( 2(d+1) + \frac d{q-2} \right) \right) a_{2,0}^2 - 
		d(q-1) \mu_n \sqrt{\frac{2d\,|\Sph^d|}{(d+1)^2(d+3)}} a_{2,0}  \right. \\
		& \quad + \sum_{m\neq 2} \left( d+2 - C |\mu_n|\left( 2(d+1) + \frac d{q-2} \right) \right) a_{2,m}^2 \\
		& \quad \left. + \sum_{\ell=3}^\infty \sum_m \left( \ell(\ell+d-1) - d - C|\mu_n| \left( \ell(\ell+d-1) + \frac d{q-2}\right) \right) a_{\ell,m}^2 \right).
	\end{align*}
	Since $\mu_n\to 0$, we have for $n$ large enough, uniformly in $\ell\geq 2$,
	$$
	\ell(\ell+d-1) - d - C|\mu_n| \left( \ell(\ell+d-1) + \frac d{q-2}\right) > 0 \,.
	$$
	Under this assumption and abbreviating
	$$
	\rho_n := d+2 - C |\mu_n|\left( 2(d+1) + \frac d{q-2} \right)>0 \,,
	$$
	we can bound
	\begin{align*}
		& \|R_n\|^2 - \frac{d(q-1)}{(q-2)} \|R_n\|_2^2 - d (q-1) \mu_n \int_{\Sph^d} g^2 R_n\,d\omega - C |\mu_n| \|R_n\|^2 \\
		& \geq \rho_n a_{2,0}^2 - d(q-1) \mu_n \sqrt{\frac{2d\,|\Sph^d|}{(d+1)^2(d+3)}} a_{2,0} \\
		& = \rho_n \left( \left( a_{2,0} - \frac{d(q-1)}{2\rho_n} \sqrt{\frac{2d\,|\Sph^d|}{(d+1)^2(d+3)}} \mu_n \right)^2 
		- \frac{d^2(q-1)^2}{4\rho_n^2} \frac{2d\,|\Sph^d|}{(d+1)^2(d+3)} \mu_n^2 \right) \\
		& \geq - \frac{d^2(q-1)^2}{4\rho_n} \frac{2d\,|\Sph^d|}{(d+1)^2(d+3)} \mu_n^2 = 
		- \frac{d^2(q-1)^2}{4(d+2)} \frac{2d\,|\Sph^d|}{(d+1)^2(d+3)} \mu_n^2 + \mathcal O(|\mu_n|^3) \,.
	\end{align*}
	
	To summarize, we have shown that
	\begin{align*}
		\lambda_n^{-2} \left( \|u_n\|^2 - Y \|u_n\|_q^2 \right)
		& \geq \mu_n^4 \left( \frac{d(q-1)(q+d)}{2(d+1)^2(d+3)} - \frac{d^3(q-1)^2}{2(d+1)^2(d+2)(d+3)} \right) |\Sph^d| \\
		& \quad + \mathcal O(|\mu_n|^5) \\
		& = \mu_n^4\, \frac{d(q-1)(2d-(d-2)q)}{2(d+1)(d+2)(d+3)} |\Sph^d| + \mathcal O(|\mu_n|^5) \,.
	\end{align*}
	
	On the other hand, we have, by the orthogonality conditions and \eqref{eq:sphcomp},
	\begin{align*}
		\mu_n^4 & = \frac{\|u_n - \overline{u_n} \|^4}{\lambda_n^4 (\|g\|^2 + \|R_n\|^2)^2}  
		= \frac{(d+1)^2}{(q-1)^2 Y^2}\, \|u_n - \overline{u_n}\|^4  \left( 1+ o(1) \right).
	\end{align*}
	Inserting this into the previous bound, we get the claimed asymptotic inequality. This completes the proof in the case $4\leq q<2d/(d-2)$.
	
	\medskip
	
	\emph{Step 4.} In the remainder of the proof we deal with the technical problems arising in the case where $q<4$. Just like in the proof of Lemma \ref{seconddecompenergymfd}, the problem is the expansion of $|1+\tau|^q$ to fourth order in $\tau$, for which we need $\mu_n(g+R_n)$ to tend to zero in $L^\infty$. While this may, in general, not be the case, in the proof of Lemma \ref{seconddecompenergymfd} we got around this problem by noting that the $L^\infty$ convergence holds for the spherical mean and the remainder can be controlled by a spectral gap estimate.
	
	In the present situation we will try to adapt the same proof and also argue by integrating out variables, but the new difficulty will be that the resulting one-dimensional function does not converge in $L^\infty$ uniformly over its interval of definition. This problem can be overcome by dealing with the boundary and the bulk separately.
	
	To be more specific, consider $u_n$ as in \eqref{eq:seconddecompsphereproof} and then further decompose
	$$
	R_n = S_n + T_n
	\qquad\text{with}\qquad
	S_n(\omega_{d+1}) := |\Sph^{d-1}|^{-1} \int_{\Sph^{d-1}} R_n(\sqrt{1-\omega_{d+1}^2}\theta,\omega_{d+1})\,d\theta \,.
	$$
	In words, $S_n$ is obtained from $R_n$ by averaging over the spheres $\{ ((1-\omega_{d+1}^2)^{1/2}\theta,\omega_{d+1})\in\Sph^d:\ \theta\in\Sph^{d-1} \}$ orthogonal to the $e_{d+1}$-axis, parametrized by their height $\omega_{d+1}$. The function $T_n$ has the property that for any function $\phi$ of $\omega_{d+1}$ alone,
	\begin{equation}
		\label{eq:meanzerotsphere}
		\int_{\Sph^d} \phi(\omega_{d+1}) T_n\,d\omega = 0 \,.
	\end{equation}
	By orthogonality,
	$$
	\| R_n \|^2 = \|S_n\|^2 + \|T_n\|^2 \,,
	$$
	so $\|R_n\|\to 0$ implies $\|S_n\|\to 0$. The difficulty compared to the proof of Lemma \ref{seconddecompenergymfd} is that this does not imply that $\|S_n\|_\infty \to 0$. To be more explicit,
	$$
	\|S_n\|^2 = |\Sph^{d-1}| \int_0^\pi \left( (\partial_\theta (S_n(\cos\theta)))^2 + \frac{d}{q-2} S_n(\cos\theta)^2 \right) \sin^{d-1}\theta\,d\theta
	$$
	and we note that the weight $\sin^{d-1}\theta$ degenerates at the boundary $\theta\in\{0,\pi\}$. Before dealing with this problem, we get rid of the term $T_n$ essentially in the same way as in the proof of Lemma \ref{seconddecompenergymfd}.
	
	\medskip
	
	\emph{Step 5.} Let us set
	$$
	u_n = \tilde u_n + \lambda_n \mu_n T_n
	\qquad\text{with}\qquad
	\tilde u_n := \lambda_n \left( 1 + \mu_n (g + S_n) \right).
	$$
	Then by an expansion to second order, similarly as before,	
	\begin{align*}
		& \lambda_n^{-2} \left( \|u_n\|^2 - Y \| u_n\|_q^2 \right) = \lambda_n^{-2} \left( \|\tilde u_n\|^2 - Y \|\tilde u_n\|_q^2 \right) \\
		& \quad + \mu_n^2 \left( \|T_n\|^2  - (q-1) Y |\Sph^d|^{-1+2/q} \int_{\Sph^d} T_n^2\,d\omega + \mathcal O(|\mu_n|^{\min\{1,q-2\}} \|T_n\|_q^2) \right).
	\end{align*}
	Just like in the proof of Lemma \ref{seconddecompsphere}, the term quadratic in $T_n$ involves the operator $-\Delta_{\Sph^d} -d$. The kernel of this operator is spanned by $\omega_1,\ldots,\omega_{d+1}$ and its negative spectral subspace is spanned by constants. We claim that $T_n$ is orthogonal to all these functions. For constants and $\omega_{d+1}$ this follows from \eqref{eq:meanzerotsphere}, and for $\omega_1,\ldots,\omega_d$ it follows from the fact that both $R_n$ and $S_n$ are orthogonal to these. As a consequence of the orthogonality relations, we have
	$$
	\|T_n\|^2  - (q-1) Y |\Sph^d|^{-1+2/q} \int_{\Sph^d} T_n^2\,d\omega \gtrsim \|T_n\|^2
	$$
	with an implicit constant depending only on $d$. Thus, if $n$ is large enough, the error term $\mathcal O(|\mu_n|^{\min\{1,q-2\}} \|T_n\|_q^2)$ can be absorbed and we conclude that
	$$
	\|u_n\|^2 - Y \| u_n\|_q^2 \geq \|\tilde u_n\|^2 - Y \|\tilde u_n\|_q^2 \,.
	$$
	Continuing to argue as in the proof of Lemma \ref{seconddecompenergymfd} we arrive at
	$$
	\frac{\| u_n\|^2 \left( \| u_n\|^2 - Y\| u_n\|_q^2\right)}{\|  u_n - \overline{ u_n}\|^4}
	\geq \frac{\| \tilde u_n\|^2 \left( \|\tilde u_n\|^2 - Y\|\tilde u_n\|_q^2\right)}{\| \tilde u_n - \overline{\tilde u_n}\|^4}  \left( 1 + o(1)\right).
	$$
	This accomplishes our goal of removing the term $T_n$.
	
	\medskip
	
	\emph{Step 6.} It remains to deal with the failure of $L^\infty$ convergence of $S_n$. We first assume that $d>2$ to present the argument in the cleanest way. Then, for any function $v\in H^1(\Sph^d)$ that depends only on $\omega_{d+1}$,
	\begin{equation}
		\label{eq:linftysphere}
		|v(\omega)| \lesssim \delta(\omega)^{-(d-2)/2} \|v\| \,,
		\ \text{where}\ \delta(\omega) := \dist(\omega,\!\{ (0,\ldots,0,+1),\! (0,\ldots,0,-1)\}\!) \,.
	\end{equation}
	This follows, for instance, from the well-known inequality, valid for all radial $w\in\dot H^1(\R^d)$,
	$$
	|w(x)| \lesssim |x|^{-(d-2)/2} \|\nabla w\|_2 \,.
	$$
	Indeed, in obvious notation,
	$$
	|w(r)| = \left| \int_r^\infty w'(s)\,ds \right| \leq \left( \int_r^\infty s^{-d+1}\,ds \ \int_r^\infty (w'(s))^2 s^{d-1}\,ds \right)^{1/2}.
	$$
	This implies \eqref{eq:linftysphere} either by a localization argument or by stereographic projection.
	
	As a consequence of \eqref{eq:linftysphere}, there is a constant $C>0$, depending only on $d$, such that if $\delta(\omega)\geq C |\mu_n|^{2/(d-2)}$, then $|\mu_n(g(\omega)+ S_n(\omega))| \leq 1/2$. (Here we also used $\|g+S_n\|\lesssim 1$.) Thus, if we set
	$$
	\mathcal C := \left\{ \omega\in\Sph^d:\ \delta(\omega)< C |\mu_n|^{2/(d-2)} \right\},
	$$
	then, by the same arguments as in the proof of Lemma \ref{seconddecompenergy},
	\begin{align*}
		& |\lambda_n|^{-q} \int_{\Sph^d\setminus\mathcal C} |\tilde u_n|^q\,d\omega
		= \int_{\Sph^d\setminus\mathcal C} \left( 1+ q\mu_n(g+S_n) + \frac12 q(q-1)\mu_n^2 (g+S_n)^2 \right)d\omega \\
		& \qquad + \int_{\Sph^d\setminus\mathcal C} \left( \frac16 q(q-1)(q-2) \mu_n^3 (g^3 + 3 g^2 S_n) + \frac1{24} q(q-1)(q-2)(q-3) \mu_n^4 g^4 \right)d\omega \\
		& \qquad + \mathcal O(|\mu_n|^3 \|S_n\|^2 + |\mu_n|^5) \,.
	\end{align*}
	On the other hand, in $\mathcal C$ we expand to second order,
	\begin{align*}
		& |\lambda_n|^{-q} \int_{\mathcal C} |\tilde u_n|^q\,d\omega
		= \int_{\mathcal C} \left( 1+ q\mu_n(g+S_n) + \frac12 q(q-1)\mu_n^2 (g+S_n)^2 \right)d\omega \\
		& \qquad + \mathcal O\left( |\mu_n|^{\min\{3,q\}} \int_{\mathcal C} |g+S_n|^{\min\{3,q\}}\,d\omega
		+ |\mu_n|^{q} \int_{\mathcal C} |g+S_n|^{q}\,d\omega \right).
	\end{align*}
	Let us bound the remainder term. We let $1\leq p \leq 2^*=2d/(d-2)$. (We will later choose $p=\min\{3,q\}$ and $p=q$.) Using the fact that $g$ is bounded and that $H^1$ embeds into $L^{2^*}$, we obtain
	\begin{align*}
		|\mu_n|^p \!\int_{\mathcal C} |g+S_n|^p \,d\omega
		& \lesssim |\mu_n|^p \!\left( \int_{\mathcal C} |g|^p\,d\omega +\! \int_{\mathcal C} |S_n|^p\,d\omega\! \right)
		\lesssim |\mu_n|^p \left( |\mathcal C| + |\mathcal C|^{(2^*-p)/2^*} \|S_n\|_{2^*}^p \right) \\
		& \sim |\mu_n|^{p+2^*} + |\mu_n|^{2^*} \|S_n\|_{2^*}^p \,.
	\end{align*}
	Here we used $|\mathcal C| \sim |\mu_n|^{2^*}$. A similar argument shows that
	\begin{align*}
		& \int_{\mathcal C} \left( \frac16 q(q-1)(q-2) \mu_n^3 (g^3 + 3 g^2 S_n) + \frac1{24} q(q-1)(q-2)(q-3) \mu_n^4 g^4 \right)d\omega \\
		& = \mathcal O( |\mu_n|^{3+2^*} + |\mu_n|^{2+2^*} \|S_n\|_{2^*}) \,.
	\end{align*}
	To summarize, we have
	\begin{align*}
		& |\lambda_n|^{-q} \int_{\mathcal C} |\tilde u_n|^q\,d\omega
		= \int_{\mathcal C} \left( 1+ q\mu_n(g+S_n) + \frac12 q(q-1)\mu_n^2 (g+S_n)^2 \right)d\omega \\
		& \qquad + \int_{\mathcal C} \left( \frac16 q(q-1)(q-2) \mu_n^3 (g^3 + 3 g^2 S_n) + \frac1{24} q(q-1)(q-2)(q-3) \mu_n^4 g^4 \right)d\omega \\
		& \qquad + \mathcal O( |\mu_n|^{\min\{3,q\}+2^*} + |\mu_n|^{2^*} \|S_n\|_{2^*}^{\min\{3,q\}} + |\mu_n|^{2+2^*} \|S_n\|_{2^*} ) \,. 
	\end{align*}
	Adding this to the expansion on $\Sph^d\setminus\mathcal C$ and using the orthogonality conditions, we finally obtain
	\begin{align*}
		|\lambda_n|^{-q} \|\tilde u_n\|_q^q & = |\Sph^d| + \frac12 q(q-1) \mu_n^2 \left( \|g\|_2^2 + \|S_n\|_2^2 \right) + \frac12 q(q-1)(q-2) \mu_n^3 \int_{\Sph^d} g^2 S_n\,d\omega \\
		& \quad + \frac1{24} q(q-1)(q-2)(q-3) \mu_n^4 \| g\|_4^4 \\
		& \quad + \mathcal O( |\mu_n|^3 \|S_n\|^2  + |\mu_n|^{2^*} \|S_n\|_{2^*}^{\min\{3,q\}}
		+ |\mu_n|^5 + |\mu_n|^{\min\{3,q\}+2^*} ) \,.
	\end{align*}
	Here we slightly simplified the error terms, using $|\mu_n|^{2+2^*} \|S_n\|_{2^*}\lesssim |\mu_n|^3 \|S_n\|^2 + |\mu_n|^{1+2\cdot 2^*}$ and $1+2\cdot 2^*> 5$.
	
	The upshot is that we have almost the same bound as in the case $q\geq 4$, except that the remainder $|\mu_n|\|R_n\|^2$ there is now replaced by $|\mu_n|^3 \|S_n\|^2  + |\mu_n|^{2^*} \|S_n\|_{2^*}^{\min\{3,q\}}$ and the remainder $|\mu_n|^5$ there is now replaced by $|\mu_n|^5 + |\mu_n|^{\min\{3,q\}+2^*}$. These replacements, however, do not affect the proof. Indeed, the only thing that was important about the first remainder was that it was $o(\mu_n^2) \|R_n\|^2$ and about the second remainder that is was $o(\mu_n^4)$. This is satisfied in the present case and therefore one can proceed in the same way as before. 
	
	\medskip
	
	\emph{Step 7.} Finally, we briefly address the necessary changes for $d=2$. In this case, inequality \eqref{eq:linftysphere} holds only with $\delta(\omega)^{-\alpha}$ for arbitrarily small $\alpha>0$, but not with $\alpha=0$. Moreover, $H^1$ is embedded into $L^r$ for arbitrary large $r<\infty$, but not for $r=\infty = 2^*$. Thus, if one follows the above proof, these two issues imply that the remainder estimates in the expansion of $\|\tilde u_n\|_q^q$ become worse by a factor $|\mu_n|^{-\epsilon}$ for arbitrarily small $\epsilon>0$. This, however, is still enough to conclude the proof along the same lines.
\end{proof}

\begin{remark}\label{optimality3}
	The bound in Lemma \ref{seconddecompenergysphere} is best possible, both with respect to the power four and with respect to the constant on the right side. Indeed, it is saturated as $\epsilon\to 0$ for $u_\epsilon = 1 + \epsilon g + \epsilon^2 h$ with $h(\omega_{d+1}) := (d(q-1)/(2(d+2))) (\omega_{d+1}^2 -1/(d+1))$. In the notation of the previous proof, this corresponds to $\mu_\epsilon=\epsilon$ and $R_\epsilon = \epsilon (h+o(1))$. The function $h$ is chosen in such a way that the square that is completed in the previous proof (Step 3) vanishes to leading order.	
\end{remark}

Theorem \ref{main3} follows from Lemma \ref{seconddecompenergysphere} in the same way as Theorem \ref{main} follows from Lemma \ref{seconddecompenergy}. We omit the details.



\bibliographystyle{amsalpha}

\end{document}